\documentclass{amsart}

\usepackage{lineno}
\usepackage[all]{xy}
\usepackage{amsmath}
\usepackage{amssymb}
\usepackage{array}
\usepackage{color}
\usepackage{verbatim}

\usepackage[
    colorlinks=true,
    linkcolor=blue,
    citecolor=blue,
    urlcolor=blue
]{hyperref}

\newtheorem{theorem}{Theorem}[section]
\newtheorem{lemma}[theorem]{Lemma}

\theoremstyle{definition}
\newtheorem{definition}[theorem]{Definition}

\theoremstyle{remark}
\newtheorem{remark}[theorem]{Remark}

\def\Q{{\mathbb Q}}
\def\Der{\mathrm{Der}}
\def\Hom{\mathrm{Hom}}

\def\cat0{\mathrm{cat}\_0}

\def\ker{\mathrm{ker}}

\def\B{B\mathrm{aut}_1}

\numberwithin{equation}{section}

\begin{document}
\author{Yang Bai}
\address{School of Mathematical Sciences, Nankai University, Tianjin 300071, P.R.China}
\email{1687694204@qq.com}

\author{Xiugui Liu}
\address{School of Mathematical Sciences and LPMC, Nankai University, Tianjin 300071, P.R.China}
\email{xgliu@nankai.edu.cn}

\author{Jiaxi Zha}
\address{School of Mathematical Sciences, Nankai University, Tianjin 300071, P.R.China}
\email{1093913699@qq.com}

\thanks{The second author was supported in part by the National Natural Science Foundation of China (Grant No. 12171165).}
\subjclass[2010]{55P62}
\date{}
\title[Rational Realization of Classifying Spaces]{On The Rational Realization of Even-dimensional Spheres and Products of Eilenberg--MacLane Spaces as Classifying Spaces}

\begin{abstract}
In this paper, we study the rational realization problem for the classifying space $\B(X)$. We prove that if $S^{2n}$ is realized as $\B(X)$ for a simply-connected space $X$, then $X$ is $\pi$-infinite and has vanishing rational Gottlieb elements above degree $2n-1$. In particular, even-dimensional spheres cannot be realized as $B\mathrm{aut}_1(X)$ for any simply-connected $\pi$-finite space $X$. We also prove that, for all $n\geq 2$ and $s,t\geq 1$, the product of Eilenberg--MacLane spaces $K(\Q^s,n)\times K(\Q^t,n+1)$ cannot be realized as $B\mathrm{aut}_1(X)$ for any simply-connected $\pi$-finite space $X$. Moreover, we prove that if $r\geq 2$ and $n\geq 3$ and $K(\Q^r,n)$ is realized as $\B(X)$ for a $\pi$-finite space $X$, then $X\simeq_{\Q}K(\Q^r,n-1)$. The proofs are based on two structural results for Gottlieb elements in the derivation Lie algebra of a Sullivan minimal model, which provide a uniform method for these realization problems.
\end{abstract}

\keywords{Rational homotopy theory, classifying spaces, Sullivan minimal models, Lie models, Gottlieb elements, derivations.}

\maketitle

\section{Introduction}

The classification theory for fibrations was developed in a series of foundational papers \cite{D,M,J}. For a fixed CW complex $X$, the universal fibration $X\rightarrow UE \rightarrow B{\rm aut}(X)$ establishes a one-to-one correspondence from the set of fiber homotopy types of fibrations over $B$ with fibre $X$ to the set of homotopy classes of maps, $[B,B{\rm aut}(X)]$. In this paper, we will consider the classifying space $B{\rm aut}_1(X)$, which classifies fibrations with simply-connected base space. It can be also considered as the universal cover of $B{\rm aut}(X)$. 

The space $\B(X)$ was among the first objects described in rational homotopy theory, and the algebraic models for $B{\rm aut}_1(X)$ were given by many authors (cf. \cite[Ch.7]{T}). Sullivan gave a Lie model for $B{\rm aut}_1(X)$ in terms of derivations of the Sullivan minimal model \cite{S}. Schlessinger and Stasheff constructed another equivalent model for $B{\rm aut}_1(X)$ in terms of derivations of a Quillen minimal model \cite{SS}. 

A long-standing and open problem in rational homotopy theory was raised by Schlessinger in 1976 (see \cite[p.519]{F-H-T}), which seeks to characterize the space which can be realized as the classifying space $\B (X)$ for some space $X$ up to rational homotopy equivalence. 

Several results have been given in recent years with the assumption of $X$ $\pi$-finite. Recall that a topological space is called $\pi$-finite if it has only finitely many non-zero homotopy groups. Otherwise the space is called $\pi$-infinite. Lupton and Smith \cite{J-B, 2019L-S} constructed some infinite families of spaces which can be realized as classifying spaces, and also showed that $\mathbb{C}P^2$, $\mathbb{C}P^3$, $\mathbb{C}P^4$ and $S^4$ are not realized as $\B(X)$ for $X$ simply-connected and $\pi$-finite. Later, Bai, Liu and Xie \cite{BY} proved the space with Sullivan minimal model of the form $(\Lambda(a, b), db=a^p)$ for $|b|\leq 5n+2$ cannot be realized as $\B(X)$ for $X$ $n$-connected and $\pi$-finite. 

Subsequently, Lupton and Smith \cite{GE} studied simply-connected space $X$ satisfying $\B(X)\simeq _{\Q} S^{2n}$ for $n=1,2,3$ without the assumption of $X$ $\pi$-finite. They proved that such $X$ must be $\pi$-infinite with vanishing rational Gottlieb elements above degree $2n-1$. Recently, Bai and Liu \cite{BY1} generalized their result to the case $n=4,5$, and proved the non-realization of $\B (X)\simeq_{\mathbb{Q}} K(\Q^s,n) \times K(\Q^t,n+1)$ with $2\leq n\leq 10$ and $s,t\geq 1$ for $\pi$-finite space. It should be noted that, so far, all the non-realization results are given within the restricted dimensions. 

The aim of the present paper is to remove these dimension restrictions by a uniform method. Our approach is based on the derivation Lie model of $\B(X)$ given by Sullivan \cite{S}. We prove two structural results for Gottlieb elements in the derivation Lie algebra of a Sullivan minimal model. The first produces compatible derivations from a Gottlieb element, while the second shows that two Gottlieb elements impose strong restrictions on their degrees and on the generators of the minimal model. These results provide the main technical input for all of our applications and allow us to treat several realization problems simultaneously in arbitrary degree. In particular, our extension of the previous results to arbitrary degrees does not rely on a degree-by-degree analysis. Moreover, the same results provide useful tools for several further realization problems, which will be studied in forthcoming work. 

As applications, we obtain the following extensions of the above results to all degrees. 

\begin{theorem}\label{thm1}
    Let $X$ be a simply-connected space with rational homology of finite type. Suppose that
    $$\B (X)\simeq_{\mathbb{Q}} S^{2n}.$$
    Then $X$ is $\pi$-infinite with vanishing rational Gottlieb elements above degree $2n-1$. In particular, the rational homotopy type of $S^{2n}$ cannot be realized as the classifying space of any $\pi$-finite space. 
\end{theorem}

\begin{remark}
	In forthcoming work, we show that for every $n\geq 1$, $S^{2n}$ can be realized as $\B(X)$ for some simply-connceted $\pi$-infinite space $X$. Thus the $\pi$-finiteness condition imposes a strong restriction on the realization problem.
\end{remark}

Our second application gives a non-realization result for products of Eilenberg--MacLane spaces in consecutive degrees.

\begin{theorem}\label{thm2}
    The rational homotopy type of $K(\Q^s,n)\times K(\Q^t,n+1)$ for $n\geq 2$ and $s,t\geq 1$ cannot be realized as the classifying space of any simply-connected $\pi$-finite space with rational homology of finite type. 
\end{theorem}

A second aspect of the realization problem is rigidity: once a rational homotopy type $Y$ is realizable as $\B(X)$, to what extent is the rational homotopy type of $X$ determined by $Y$? This problem was considered in \cite{Y} by Yamaguchi, who determined all elliptic spaces $X$ with $B{\rm aut}_1(X)$ rank one. Later, Bai, Liu and Xie \cite{BY} proved that if $K(\Q^r,n)$ with $r\geq2$ can be realized as $\B (X)$ for elliptic space $X$, then $n$ is even and $X$ has the rational homotopy type of $\prod_r S^{n-1}$. The following theorem generalizes the above result to $\pi$-finite spaces.

\begin{theorem}\label{thm3}
    Let $X$ be a simply-connected $\pi$-finite space with rational homology of finite type. Suppose that $$\B(X)\simeq_{\Q}K(\Q^{r},n)$$ for $n\geq 3$ and $r\geq 2$. Then $X\simeq_{\Q} K(\Q^{r},n-1).$
\end{theorem}

\subsection*{Convention.} Unless otherwise specified, all spaces appearing in this paper are assumed to be simply-connected and have rational homology of finite type.

\subsection*{Organization.} We start by recalling basic facts about rational homotopy theory, classifying spaces and Gottlieb elements in Section 2. Then in Section 3, we establish two structural lemmas on derivation cycles which form the main technical input of the paper. In Section 4, we finally prove Theorems \ref{thm1}, \ref{thm2} and \ref{thm3}.

\section{Preliminaries}

In this section, we first recall basic facts in rational homotopy theory and the results about the classifying space $B{\rm aut}_1(X)$. The standard reference is \cite{F-H-T}.

The Sullivan minimal model of a simply-connected space $X$ with rational homology of finite type is a free commutative differential graded algebra (abbr. CDGA) $(\Lambda V,d)$ with a graded vector space $V=\bigoplus_{i\geq 2}V^i$ where ${\rm dim} V^i< \infty$ and a decomposable differential, i.e.,
$$d(V^i)\subset (\Lambda^+ V\cdot \Lambda^+ V)^{i+1}\ {\rm and}\ d\circ d=0. $$
Here $\Lambda^+ V$ is the augmentation ideal of $\Lambda V$. The degree of an element $x$ is denoted by $|x|$. Then we have
$$xy=(-1)^{|x||y|}yx\ {\rm and}\ d(xy)=d(x)y+(-1)^{|x|}xd(y)$$ for $x, y\in \Lambda V$. Note that there exist the following isomorphisms:
$${\rm Hom}(V^i, \mathbb{Q})\cong \pi_{i}(X)\ {\ \rm and \ }\ H^{\ast}(\Lambda V, d)\cong H^{\ast}(X; \mathbb{Q}).$$

A derivation $\theta$ that reduces degree by $n$ is a linear map $$\theta: \Lambda V\rightarrow \Lambda V$$ of degree $-n$ such that $$\theta(xy)=\theta(x)y+(-1)^{n|x|}x\theta(y)$$ for $x, y\in \Lambda V$. Let $({\rm Der}(\Lambda V),D)$ denote the differential graded Lie algebra (abbr. DGL) of derivations of $\Lambda V$, where the Lie bracket is given by $$[\theta_1,\theta_2]=\theta_1\theta_2-(-1)^{|\theta_1||\theta_2|}\theta_2\theta_1,$$
and the differential $D$ is given by $$D(\theta)=[d,\theta]$$ for $\theta\in {\rm Der}(\Lambda V)$. For convenience, we denote by ${\rm Der}_k(\Lambda V)$ the set of derivations that reduce degree by $k$. In particular, we remark that $(\Der(\Lambda V),D)$ can be viewed as a right $(\Lambda V,d)$-module by setting
$$\theta x(v)=(-1)^{|x||v|}\theta(v)x$$
for $\theta\in \Der(\Lambda V)$ and $x,v\in \Lambda V$, where the module action is given by
\begin{align*}
    \begin{split}
        \Der_k(\Lambda V) \times(\Lambda V)^n &\rightarrow \Der_{k-n}(\Lambda V) \\
        (\theta,x)\ \ \ \ \ \ \ &\mapsto \theta x.
    \end{split}
\end{align*}
It is worth emphasizing that $\theta x$ or $\theta\cdot x$ denotes the action of $x$ on the derivation $\theta$, whereas $\theta(x)$ denotes the image that the derivation $\theta$ sends $x$ to. 

Note that the wordlength is another gradation in $\Lambda V$. In the following decomposition
$$\Lambda V=\bigoplus_{l\geq0}\Lambda^l V$$
with $\Lambda^0 V=\Q$, we call the elements in $\Lambda^l V$ have wordlength $l$. For any $x\in \Lambda^+ V$, $x$ can be decomposed as
$$x=x_1+x_2+\dots,$$
where $x_l\in \Lambda^l V$ for $l\geq 1$. We denote $x_l$ by $L_{l}(x)$ in the above decomposition of $x$. Similarly, by the isomorphism $\Der (\Lambda V)\cong \Hom (V,\Lambda V)$, wordlength in $\Lambda V$ also induces a gradation in $\Der (\Lambda V)$, i.e.,
$$\Der(\Lambda V)\cong \bigoplus_{l\geq -1}\Hom (V,\Lambda ^{l+1}V). $$
Thus for any $\theta\in \Der(\Lambda V)$, $\theta$ can be decomposed as
$$\theta=\theta_{-1}+\theta_0+\dots,$$
where $\theta_l\in \Hom (V,\Lambda ^{l+1}V)$ is a derivation that increases wordlength by $l$. We denote $\theta_l$ by $L_{l}(\theta)$ in the decomposition of $\theta$. Note that the differential $d$ in $\Lambda V$ can be decomposed by
$$d=d_1+d_2+\dots,$$
where $d_l=L_l(d)$. Therefore the differential $D=[d,-]$ in $\Der(\Lambda V)$ increases wordlength at least by 1.

Let ${\rm Der}'(\Lambda V)$ denote the sub Lie algebra of ${\rm Der}(\Lambda V)$ defined by
$${\rm Der}'_k(\Lambda V)=\left\{
\begin{array}{ll}
{\rm Der}_k(\Lambda V),&k>1,\\
Z({\rm Der}_1(\Lambda V)),&k=1,\\
0,&k<1,
\end{array}
\right.$$
where $Z({\rm Der}_1(\Lambda V))$ is the subspace of all cycles in ${\rm Der}_1(\Lambda V)$. We recall Sullivan's original result about the model for the classifying space $B{\rm aut}_1(X)$.

\begin{theorem}\label{liu}\rm{(\cite[p. 313]{S})}
Let $X$ be a space with Sullivan minimal model $(\Lambda V, d)$, then the DGL $({\rm Der}'(\Lambda V),D)$ is a Lie model for $B{\rm aut}_1(X)$.
\end{theorem}

Recall that the homotopy groups of the loop space $\Omega B{\rm aut}_1(X)$ admit a natural bilinear pairing $[-,-]$ called the \textit{Samelson product} (see \cite[Ch. \uppercase\expandafter{\romannumeral3}]{W}). In particular, we have the following theorem.

\begin{theorem}\label{J}
Let $X$ be a space
with Sullivan minimal model $(\Lambda V, d)$. Then there is an isomorphism of graded Lie algebras
$$\pi_{\ast}(\Omega B{\rm aut}_1(X)) \cong H_{\ast}({\rm Der}'(\Lambda V),D)$$
where the left-hand graded space has the Samelson product. 
\end{theorem}

In the previous work of Lupton and Smith \cite{J-B, 2019L-S}, they studied the realization problem for $\pi$-finite spaces, which depends heavily on the following theorem.

\begin{theorem}\label{pf}\rm{(\cite[Prop. 2.2]{J-B})}
Suppose $X$ is $\pi$-finite with
$$\pi_i(X) =\left\{
\begin{array}{ll}
\mathbb{Q}^r\ {\rm some} \ r \geq 1, & i = N,\\
0, &i > N.
\end{array}
\right.$$
Then we have
$$\pi_i(B{\rm aut}_{1}(X)) =\left\{
\begin{array}{ll}
\mathbb{Q}^r,& i = N+1,\\
0, &i > N+1.
\end{array}
\right.$$
\end{theorem}

Finally, we recall some basic facts about Gottlieb elements, which were first introduced by Gottlieb \cite{G} with a different name. A homotopy class $\alpha\in \pi_n(X)$ is called a \textit{Gottlieb element for a space $X$} if the map $\alpha \vee 1_X:S^n\vee X\rightarrow X$ can be extended to a map $\varphi_{\alpha}:S^n\times X\rightarrow X$ \cite[p.377]{F-H-T}. The Gottlieb elements for $X$ form a subgroup $G_\ast(X)\subset \pi_\ast(X)$. The element in $G_\ast(X_\Q)$ is called the \textit{rational Gottlieb element for $X$}. A \textit{Gottlieb element for a Sullivan minimal algebra $(\Lambda V,d)$} is a linear map $f:V^n\rightarrow \Q$ which can be extended to a cycle $\theta\in \Der(\Lambda V)$, i.e., $d\theta =(-1)^n\theta d$ \cite[p.392]{F-H-T}. The Gottlieb elements for $(\Lambda V,d)$ form a graded subspace $G_\ast(\Lambda V,d)\subset \Hom(V,\Q)$. About Gottlieb elements defined for spaces and Sullivan models, we have the following result.
\begin{theorem}\cite[Prop. 29.8]{F-H-T}\label{GEtg}
    Suppose that $(\Lambda V,d)$ is a Sullivan minimal model for a simply-connected space $X$ with rational homology of finite type. Then there is an isomorphism
    $$G_\ast(\Lambda V,d) \xrightarrow{\cong} G_\ast(X_\Q).$$
\end{theorem}
\noindent In other words, Gottlieb elements of degree $k$ for a rational space $X$ can be identified with the nonzero restrictions on $V^k$ of cycles in $\Der^\prime_k(\Lambda V)$, where $(\Lambda V,d)$ is a Sullivan minimal model for $X$ (also see \cite[Thm.3.5]{2007J-B}).

\section{Structural results on derivation cycles}

Let $(\Lambda V,d)$ be a Sullivan minimal algebra with $V=V^{\geq2}$ of finite type. Inductively, define an increasing filtration $\{V(i)\}_{i\geq 0}$ of $V$ by $$V(0)=\ker d\cap V,\qquad V(i+1)=d^{-1}(\Lambda V(i))\cap V.$$
This filtration is exhaustive by the definition of Sullivan minimal model. Moreover, we can write the filtration as a direct sum $V=\oplus_{i\geq 0} V_i$
$$V_0=V(0),\ \ V_0\oplus \dots \oplus V_i=V(i). $$
Set $W=V_0\oplus V_1$, denote the suspension of $W$ by $\overline{W}$. We can define a CDGA 
$$(\Lambda \overline{W}\otimes \Lambda W,d_W),$$ 
and the differential is given by $sd_W+d_Ws=\mathrm{id}$, where $s$ is a derivation given by desuspension $s\overline{v}=v$, $sv=0$ for $\overline{v}\in \overline{W}$, $ v\in W$. 
In particular, we have 
$$d_W\overline{v}=v-sdv,\ \  d_Wv=dv\ \  \textrm{ for } \overline{v}\in \overline{W}, \ v\in W.$$
Note that $(\Lambda \overline{W}\otimes \Lambda W,d_W)$ is an acyclic CDGA \cite[Proposition 14.13]{F-H-T}. For convenience, we call $(\Lambda \overline{W}\otimes \Lambda W,d_W)$ the {\it acyclic closure} of $(\Lambda W,d)$  \cite[p. 192]{F-H-T}. 

Recall that $L_l(x)$ denotes the summand of $x \in \Lambda^+V$ which has wordlength $l$, and $L_l(\theta)$ denotes the summand of $\theta\in \Der(\Lambda V)$ which increases wordlenth by $l$. 
\begin{definition}
    For an integer $k$, we say {\it $\Der^\prime(\Lambda V)$ satisfies the condition $\mathrm{(G_k)}$}, if any cycle $ \theta \in \Der^\prime_k(\Lambda V)$ with $L_{-1}(\theta)=0$ is an exact cycle. 
\end{definition}
\noindent Note that $H_k(\Der^\prime(\Lambda V))=0$ implies that $\Der^\prime(\Lambda V)$ satisfies condition $\mathrm{(G_k)}$. 

With notation as above, we have the following lemmas. 
\begin{lemma}\label{lm1}
    Suppose that $\theta\in \Der^\prime(\Lambda V)$ is a cycle of degree $n$ with $L_{-1}(\theta)\neq 0$, and $\Der^\prime(\Lambda V)$ satisfies the condition $\mathrm{(G_k)}$ for $n-m\leq k\leq n-2$, where $m$ is an integer satisfying $2\leq m\leq n-1$. Then there exists a right $\Lambda W$-module map of degree $n$
    $$ \Phi_{\theta}: (\Lambda \overline{W})^{\leq m-1}\otimes \Lambda W  \rightarrow \Der(\Lambda V) $$
    with 
    $$\Phi_{\theta}(1)=\theta, \ \ D\Phi_{\theta}(x)=(-1)^n\Phi_{\theta}(d_Wx), \ \ \textrm{ for } \ x\in (\Lambda \overline{W})^{\leq m-1}\otimes \Lambda W. $$
    Moreover, the restriction of $L_{-1}\circ\Phi_{\theta}$ on $(\Lambda \overline{W})^{\leq m-1}$ is injective. 
    
\end{lemma}

\begin{proof}
    To define the right $\Lambda W$-module map $\Phi_{\theta}$, it is sufficient to give the definition of $\Phi_{\theta}$ on $(\Lambda \overline{W})^{\leq m-1}$. Recall that $W=V_0\oplus V_1$, and thus $\overline{W}=\overline{V_0}\oplus \overline{V_1}$. We first define $\Phi_{\theta}$ on $(\Lambda \overline{V_0})^{\leq m-1}\otimes \Lambda W$ by induction on wordlength. 

    Choose a homogeneous basis $\{v_{0,i}\}_{i\in I}$ for $V_0^{\leq m}$. Recall that $\Der^\prime(\Lambda V)$ satisfies condition $\mathrm{(G_k)}$ for $n-m\leq k\leq n-2$. Thus for $i\in I$, since $(-1)^n\Phi_{\theta}(v_{0,i})$ is a cycle with 
    $$L_{-1}\big((-1)^n\Phi_{\theta}(v_{0,i})\big)= L_{-1}\big((-1)^n\theta v_{0,i}\big)=0$$ 
    and 
    $$1\leq n-m\leq   |(-1)^n\theta v_{0,i}|=n-|v_{0,i}|\leq n-2,$$
    we conclude that $(-1)^n\theta v_{0,i}$ is an exact cycle according to the definition of condition $\mathrm{(G_k)}$. Thus we can choose $\Phi_{\theta}(\overline{v_{0,i}})$ such that 
    $$D\Phi_{\theta}(\overline{v_{0,i}}) =(-1)^n\Phi_{\theta}(v_{0,i})=(-1)^n\Phi_{\theta}(d_W \overline{v_{0,i}}).$$ 
    Thus, we define $\Phi_{\theta}$ on $ \overline{V_0}^{\leq m-1}\otimes \Lambda W$. 

    Suppose that we have defined $\Phi_{\theta}$ on $(\Lambda^{\leq l-1} \overline{V_0})^{\leq m-1}\otimes \Lambda W$. For a basis element $\overline{v_{0,i_1}} \ \overline{v_{0,i_2}}  \dots \overline{v_{0,i_{l}}}$ of $(\Lambda^{l} \overline{V_0})^{\leq m-1}$, we have 
    \begin{align*}
        \begin{split}
            D \Big( (-1)^n\Phi_{\theta}\big( d_W (\overline{v_{0,i_1}} \dots \overline{v_{0,i_{l}}}) \big) \Big) &=  \sum_{s=1}^{l} (-1)^{n+\delta_s} D\Phi_{\theta}( v_{0,i_s} \overline{v_{0,i_1}} \dots \widehat{\overline{v_{0,i_s}}} \dots \overline{v_{0,i_{l}}} )  \\
            &=\sum_{s=1}^{l} (-1)^{\delta_s} \Phi_{\theta}( d_Wv_{0,i_s}\cdot \overline{v_{0,i_1}} \dots \widehat{\overline{v_{0,i_s}}} \dots \overline{v_{0,i_{l}}} ) \\
            &=\Phi_{\theta}\big( d_W (\sum_{s=1}^{l} (-1)^{\delta_s} v_{0,i_s} \overline{v_{0,i_1}} \dots \widehat{\overline{v_{0,i_s}}} \dots \overline{v_{0,i_{l}}} )\big) \\
            &=\Phi_{\theta}\big( d_W^2 (\overline{v_{0,i_1}} \dots \overline{v_{0,i_{l}}}) \big)  \\
            &=0 ,
        \end{split}
    \end{align*}
    where $\delta_s$ denotes the Koszul convention. Similarly, we can choose $\Phi_{\theta} (\overline{v_{0,i_1}} \dots \overline{v_{0,i_{l}}})$ such that 
    $$D\Phi_{\theta} (\overline{v_{0,i_1}} \dots \overline{v_{0,i_{l}}})=  (-1)^n\Phi_{\theta}\big( d_W (\overline{v_{0,i_1}} \dots \overline{v_{0,i_{l}}}) \big) . $$
    By induction, we can define $\Phi_{\theta}$ on $(\Lambda \overline{V_0})^{\leq m-1}\otimes \Lambda W$. 

    Then by direct computations, we note that 
    $$d_W(\Lambda^l \overline{V_1})\subset (\Lambda^{l-1} \overline{V_1} \otimes W )\oplus (\overline{V_0}\otimes \Lambda^{l-1} \overline{V_1} \otimes \Lambda W   )$$
    Thus by similar arguments as above, we can also define $\Phi_{\theta}$ on 
    $$(\Lambda \overline{V_0} \otimes \Lambda \overline{V_1})^{\leq m-1} \otimes \Lambda W=(\Lambda \overline{W})^{\leq m-1}\otimes \Lambda W$$ 
    by induction on wordlength of $\overline{V_1}$, which completes the definition of $\Phi_{\theta}$. According to the definition, $\Phi_{\theta}$ commutes with differentials. 

    Finally, we show the injectivity of $L_{-1}\Phi_{\theta}|_{(\Lambda \overline{W})^{\leq m-1}}$ by induction on the degree. Since $\Phi_{\theta}(1)=\theta$ and $L_{-1}(\theta)\neq 0$, $L_{-1}\Phi_{\theta}$ is injective on $(\Lambda \overline{W})^0$. Suppose that we have shown $L_{-1}\Phi_{\theta}$ is injective on $(\Lambda \overline{W})^{\leq k}$. For nonzero $ \overline{v}\in (\Lambda \overline{W})^{k+1}$, in order to show $ L_{-1}\Phi_{\theta}(\overline{v})\neq 0 $, it is sufficient to prove
    $$L_{0}\big(D \Phi_{\theta}(\overline{v}) \big)=L_{0} \Phi_{\theta} (d_W\overline{v})\neq 0. $$
    Since $L_{0} \Phi_{\theta}\big( (\Lambda \overline{W})^{\leq m-1}\otimes \Lambda^{\geq 2} W\big)=0$ by direct computations of wordlength, it follows that $L_{0} \Phi_{\theta} d_W$ coincides with the composition  
    $$(\Lambda \overline{W})^{k+1} \xrightarrow{\, d_W \, } (\Lambda \overline{W})^{\leq k} \otimes \Lambda^+ W \xrightarrow{\, \mathrm{projection}  \, } (\Lambda \overline{W})^{\leq k} \otimes  W \xrightarrow{\, L_{0} \Phi_{\theta} \, } \Der(\Lambda V). $$ 
    By the induction hypothesis, we conclude that the restriction $L_{0} \Phi_{\theta}$ on $(\Lambda \overline{W})^{\leq k} \otimes  W$ is injective. Thus, it is sufficient to show the composition
    \begin{equation}\label{d}
        (\Lambda \overline{W})^{k+1} \xrightarrow{\ d_W \ } (\Lambda \overline{W})^{\leq k} \otimes \Lambda^+ W \xrightarrow{\ \mathrm{projection}  \ } (\Lambda \overline{W})^{\leq k} \otimes  W
    \end{equation}
    is injective. 

    Let $(\Lambda \overline{W}\otimes \Lambda W,d_{1,W})$ be the acyclic closure of $(\Lambda W,d_1)$, where $d_1$ is the quadratic part of the differential $d$. By direct computations, we derive that the restriction $d_{1,W}|_{(\Lambda \overline{W})^{k+1}}$ is exactly the composition (\ref{d}). Thus we only need to show 
    $$\ker d_{1,W}\cap (\Lambda \overline{W})^{k+1}=0. $$
    Since $(\Lambda \overline{W}\otimes \Lambda W,d_{1,W})$ is an acyclic CDGA, for any $\overline{v}\in \ker d_{1,W}\cap (\Lambda \overline{W})^{k+1}$, we conclude that $\overline{v}$ is an exact cycle. On the other hand, $d_{1,W}$ increases the wordlength of $W$ by one, which implies $\overline{v}\in \Lambda \overline{W}\otimes \Lambda^{\geq 1}W$. Therefore, we conclude that $\overline{v}=0$, and the composition (\ref{d}) is injective. By induction, we finally show the injectivity of $L_{-1}\Phi_{\theta}|_{(\Lambda \overline{W})^{\leq m-1}}$. This completes the proof of Lemma \ref{lm1}.

\end{proof}

\begin{lemma}\label{lm2}
    Suppose that $\theta_1 ,\theta_2\in \Der^\prime(\Lambda V)$ are cycles  with $L_{-1}(\theta_1),L_{-1}(\theta_2)$ linearly independent, and $\Der^\prime(\Lambda V)$ satisfies the condition $\mathrm{(G_k)}$ for $1\leq k\leq |\theta_1|-2$ and $|\theta_2|-|\theta_1|\leq k\leq |\theta_2|-2$. Then we have $|\theta_1|=|\theta_2|$ and $V^{<|\theta_1|}=0$. 
\end{lemma}

\begin{proof}
    Without loss of generality, we might suppose $|\theta_1|\leq |\theta_2|$. Set
    $$L_{-1}(\theta_1)=y_1^\ast,\ \ L_{-1}(\theta_2)=y_2^\ast,$$
    for some $y_1,y_2\in V$. 

    Suppose that the restriction $d|_{V^{<|\theta_1|}}$ of the differential is non-trivial. It follows that $V_0^{<|\theta_1|}$ and $V_1^{<|\theta_1|}$ are both nonzero. Denote by $p$ the least degree  with the non-trivial differential. Thus, we can set 
    $$P=V^{<p}\subset V_0,\ \ Q=V^{p}\subset V_1. $$
    We claim that $P\oplus Q$ must contain a nonzero element of odd degree. In fact, suppose $P\oplus Q$ concentrates in even degrees. Then $0\neq dQ\subset (\Lambda P)^{p+1}$ has odd degree, a contradiction. Denote by $x$ the nonzero element in $P\oplus Q$ of odd degree. Let $r$ be the positive integer satisfying 
    $$r(|x|-1)\leq|\theta_1|-2 < (r+1)(|x|-1). $$ 
    Since $|\overline{x}|=|x|-1$ is even and $\overline{x}^r\in (\Lambda \overline{W} )^{\leq |\theta_1|-2}$, $\Phi_{\theta_1}(\overline{x}^r)$ is a well-defined derivation with $L_{-1}\Phi_{\theta_1}(\overline{x}^r)\neq 0$ by applying Lemma \ref{lm1} to $\theta_1$. Set $L_{-1}\Phi_{\theta_1}(\overline{x}^r)= u^\ast$ for some $u\in V$. Then we have 
    $$|u|=|\theta_1|-r|\overline{x}|<(r+1)(|x|-1)+2-r(|x|-1)=|x|+1,$$
    which implies $|u|\leq |x|\leq p$ and thus $u\in P\oplus Q\subset W$. Recall that $y_1^\ast,y_2^\ast$ are linearly independent, and thus $y_1^\ast(y_2)=0$. Then by direct computations, we can conclude that
    \begin{align*}
        \begin{split}
            (-1)^{|\theta_1|}\big(D\Phi_{\theta_1}(\overline{x}^r) \big)(y_2)&=\big(\Phi_{\theta_1}(d_W\overline{x}^r) \big)(y_2)  \\
            &=\Big(\Phi_{\theta_1}\big((x-sdx)\overline{x}^{r-1}\big) \Big)(y_2) \\
            &= \Big(\Phi_{\theta_1}(\overline{x}^{r-1})\cdot x \Big)(y_2) -\Big(\Phi_{\theta_1}\big((sdx)\overline{x}^{r-1}\big) \Big)(y_2) \ \in \Lambda^{\geq 2} V. 
        \end{split}
    \end{align*}
    This implies $u^\ast(dy_2) \in \Lambda^{\geq 2} V$ by wordlength. 

    On the other hand, since $|\overline{u}|=|u|-1\leq p-1<|\theta_1|-1$, $\Phi_{\theta_2}(\overline{u})$ is a well-defined derivation with $L_{-1}\Phi_{\theta_2}(\overline{u})\neq 0$ by applying Lemma \ref{lm1} to $\theta_2$. By direct computations, we have 
    \begin{align*}
        \begin{split}
            (-1)^{|\theta_2|}\big(D\Phi_{\theta_2}(\overline{u}) \big)(y_2)&= \big(\Phi_{\theta_2}(d_W\overline{u}) \big)(y_2) \\
            &=\big(\Phi_{\theta_2}(u-sdu) \big)(y_2) \\
            &= \big( \theta_2 \cdot u\big) (y_2) -\big(\Phi_{\theta_2}(sdu) \big)(y_2) \\
            &= (-1)^{|u||y_2|} u-\big(\Phi_{\theta_2}(u-sdu) \big)(y_2)\ \in (-1)^{|u||y_2|} u+\Lambda^{\geq 2} V. 
        \end{split}
    \end{align*}
    Set $L_{-1}\Phi_{\theta_2}(\overline{u})= w^\ast$ for some $w\in V$. It follows that 
    $$u^\ast w^\ast (dy_2) =\pm 1. $$ 
    Since $[u^\ast ,w^\ast ]=0$, we also have $w^\ast u^\ast(dy_2)= \pm 1$, which contradicts the fact $u^\ast(dy_2) \in \Lambda^{\geq 2} V$. Therefore, we finally conclude that 
    $$d|_{V^{<|\theta_1|}}=0,$$
    which implies $y_1\in V_1$. 

    Now suppose $|\theta_1|<|\theta_2|$. By direct computations, we have 
    $$\big(D \Phi_{\theta_1}(1)\big)(y_2)=\big(D\theta_1\big)(y_2)=0,$$
    which implies $y_1^\ast(dy_2) \in  \Lambda^{\geq 2} V$. Since $|\theta_1|<|\theta_2|$, $\Phi_{\theta_2}(\overline{y_1})$ is a well-defined derivation with $L_{-1}\Phi_{\theta_2}(\overline{y_1})\neq 0$. By similar arguments as above, we can obtain 
    \begin{equation*}
        (-1)^{|\theta_2|}\big(D\Phi_{\theta_2}(\overline{y_1}) \big)(y_2)= \big(\Phi_{\theta_2}(d_W\overline{y_1}) \big)(y_2) \  \in (-1)^{|y_1||y_2|}y_1 +\Lambda^{\geq 2} V, 
    \end{equation*}
    which contradicts the fact $y_1^\ast(dy_2) \in  \Lambda^{\geq 2} V$ in the same way. Thus, we reach the conclusion that $|\theta_1|=|\theta_2|$. 
    
    Finally, suppose that there exists a nonzero element $x^\prime\in V^{<|\theta_1|}$. Note that $|\theta_1|=|\theta_2|$ yields that $V^{<|\theta_1|}= V^{<|\theta_2|}\subset V_0$. Set $L_{-1} \Phi_{\theta_1}(\overline{x^\prime})=(u^{\prime})^\ast$. Then by considering 
    $$\big(D\Phi_{\theta_1}(\overline{x^\prime}) \big)(y_2) \ \ \mathrm{and} \ \ \big(D\Phi_{\theta_2}(\overline{u^\prime}) \big)(y_2), $$
    we get a contradiction by similar computations. Thus, we reach the conclusion that $V^{<|\theta_1|}=0$. This completes the proof of Lemma \ref{lm2}.

\end{proof}

\section{Proofs of the main theorems}

\begin{proof}[Proof of Theorem \ref{thm1}]

Suppose there exists a simply-connected space $X$ satisfying
$$\B (X)\simeq_{\mathbb{Q}} S^{2n},$$
with $G_{>2n-1}(X_\Q)\neq 0$. Let $(\Lambda V,d)$ be a Sullivan minimal model for $X$. By Theorem \ref{J}, we have
\begin{equation*}
\pi_{k}(\Omega S^{2n})\otimes\Q \cong H_k({\rm Der}'(\Lambda V),D)=\left\{
\begin{array}{lr}
\Q[\theta_2], &\ k=4n-2,\\
\Q[\theta_1], &\ k=2n-1,\\
0, &{\rm otherwise,\ }
\end{array}
\right.
\end{equation*}
with $[\theta_2]=\big[[\theta_1], [\theta_1]\big]$. We can suppose $\theta_2=[\theta_1, \theta_1]$. By Theorem \ref{GEtg}, we have
$$\Q L_{-1}(\theta_2)=G_{>2n-1}(\Lambda V,d)=G_{>2n-1}(X)\neq 0,$$
which implies $L_{-1}(\theta_2)\neq 0$. Since $\theta_2=[\theta_1, \theta_1]$, we can show that $L_{-1}(\theta_1)\neq 0$ by direct computations. 

Note that $\Der^\prime(\Lambda V)$ satisfies the condition $\mathrm{(G_k)}$ for $1\leq k\leq 4n-4$. By Lemma \ref{lm2}, we have $|\theta_1|=|\theta_2|$, a contradiction. Thus, we reach the conclusion that $G_{>2n-1}(X_\Q)=0$. 

Suppose that $X$ is $\pi$-finite. By Theorem \ref{pf}, we have 
$$\pi_i(X) =\left\{
\begin{array}{ll}
\mathbb{Q} , & i = 4n-2,\\
0, &i > 4n-2, 
\end{array}
\right.$$
which implies $G_{4n-2}(X_\Q)=\mathbb{Q}$. This completes the proof of Theorem \ref{thm1}.

\end{proof}

\begin{proof}[Proof of Theorem \ref{thm2}]
Suppose there exists a simply-connected and $\pi$-finite space $X$ satisfying
$$\B (X)\simeq_{\mathbb{Q}} K(\Q^s,n)\times K(\Q^t,n+1)$$
with $n\geq 2$ and $s,t\geq 1$. Let $(\Lambda V,d)$ be a Sullivan minimal model for $X$. Then by Theorems \ref{J} and \ref{pf}, we have
\begin{equation*}
 H_k({\rm Der}'(\Lambda V),D)=\left\{
\begin{array}{lr}
\Q[y_{1}^*]\oplus\dots \oplus \Q[y_{t}^*], &k=n,\ \ \ \\
\Q[\theta_1]\oplus\dots\oplus \Q[\theta_s], &k=n-1,\\
0, &{\rm otherwise,}
\end{array}
\right.
\end{equation*}
where $V^{n}=\Q \{y_1,\dots,y_t\}$ and $V^{> n}=0$. Since $X$ is simply-connected, we have $V=V^{\geq 2}$, and $\theta_1=x_1^\ast$ for some nonzero element $x_1\in V^{n-1}$ for degree reasons. Set $\theta_2=y_{1}^*$.

Note that $\Der^\prime(\Lambda V)$ satisfies the condition $\mathrm{(G_k)}$ for $1\leq k\leq n-2$. By applying Lemma \ref{lm2} to $\theta_1$ and $y_1^\ast$, we have $|\theta_1|=|y_1^\ast|$, a contradiction. This completes the proof of Theorem \ref{thm2}.

\end{proof}

\begin{proof}[Proof of Theorem \ref{thm3}]
Let $(\Lambda V,d)$ be a Sullivan minimal model for $X$. By Theorem \ref{J} and Proposition \ref{pf}, we have
\begin{equation*}
 H_k({\rm Der}'(\Lambda V),D)=\left\{
\begin{array}{lr}
\Q[y_{1}^*]\oplus\dots \oplus \Q[y_{r}^*], &k=n-1,\\
0, &{\rm otherwise,}
\end{array}
\right.
\end{equation*}
where $V^{n-1}=\Q \{y_1,\dots,y_r\}$ and $V^{> n-1}=0$. Set $\theta_1=y_1^\ast$ and $\theta_2=y_2^\ast$.

Note that $\Der^\prime(\Lambda V)$ satisfies the condition $\mathrm{(G_k)}$ for $0\leq k\leq n-3$. By Lemma \ref{lm2}, we have $V^{<|\theta_1|}=V^{<n-1}=0$. This implies $X\simeq_{ \mathbb{Q}} K(\Q^{r},n-1)$.

\end{proof}

\begin{remark}
    Lemmas \ref{lm1} and \ref{lm2}  have broader applications beyond the results considered here. Many non-realization results can also be given within the context of $\pi$-finite spaces, such as the non-realization of $S^{2n}\times K(\Q^r,2n-1)$. Further applications will be considered in our forthcoming work.
\end{remark}


\begin{thebibliography}{10}

\bibitem{BY}
Y. Bai, X. Liu, and S. Xie,
{\em Some notes on spaces realized as classifying spaces}, Topology Appl. {\bf 356} (2024), Paper No. 109030, 12 pp. 

\bibitem{BY1}
Y. Bai and X. Liu, {\em Gottlieb elements and rational homotopy types realized as classifying spaces}, Topology Appl. {\bf 373} (2025), Paper No. 109525, 19 pp. 

\bibitem{D}
A. Dold, {\em Halbexakte Homotopiefunktoren}, Lecture Notes in Mathematics, 12, Springer, Berlin-New York, 1966. 

\bibitem{F-H-T}
Y. F\'{e}lix, S. Halperin, and J.-C. Thomas, {\em Rational homotopy theory}, Graduate Texts in Mathematics, Vol. 205, Springer-Verlag, New York, 2001.

\bibitem{F-M-T}
Y. F\'elix, J. M. Moreno-Fern\'andez and D. Tanr\'e, {\em Lie models for nilpotent spaces}, Manuscripta Math. 159 (2019), 161--170.

\bibitem{G}
D. H. Gottlieb, {\em Evaluation subgroups of homotopy groups}, Amer. J. Math. {\bf 91} (1969), 729--756.

\bibitem{2007J-B}
G. Lupton and S.~B. Smith, {\em Rationalized evaluation subgroups of a map. I. Sullivan models, derivations and $G$-sequences}, J. Pure Appl. Algebra {\bf 209} (2007), no.~1, 159--171.

\bibitem{J-B}
G. Lupton and S.~B. Smith, {\em Realizing spaces as classifying spaces}, Proc. Amer. Math. Soc. {\bf 144} (2016), no.~8, 3619--3633. 

\bibitem{2019L-S}
G. Lupton and S.~B. Smith, {\em The universal fibration with fibre $X$ in rational homotopy theory}, J. Homotopy Relat. Struct. {\bf 15} (2020), no.~2, 351--368. 

\bibitem{GE}
G. Lupton and S.~B. Smith, {\em The structuring effect of a Gottlieb element on the Sullivan model of a space}, Homology Homotopy Appl. {\bf 25} (2023), no.~2, 275--296. 

\bibitem{M}
J.~P. May, {\em Classifying spaces and fibrations}, Mem. Amer. Math. Soc. {\bf 1} (1975), 1, no. 155, {\rm xiii}+98 pp.

\bibitem{SS}
M. Schlessinger and J. Stasheff, {\em Deformation theory and rational homotopy type}, arXiv:1211.1647v1. 

\bibitem{J}
J. Stasheff, {\em A classification theorem for fibre spaces}, Topology {\bf 2} (1963), 239--246. 

\bibitem{S}
D. Sullivan, {\em Infinitesimal computations in topology}, Inst. Hautes \'Etudes Sci. Publ. Math. No. 47 (1977), 269--331. 

\bibitem{T}
D. Tanr\'e, {\em Homotopie rationnelle: mod\`eles de Chen, Quillen, Sullivan}, Lecture Notes in Mathematics, Vol. 1025, Springer-Verlag, Berlin, 1983.

\bibitem{W}
G. Whitehead, {\em Elements of homotopy theory}, Graduate Texts
in Mathematics, Vol. 61, Springer-Verlag, New York, 1978. 

\bibitem{Y}
T. Yamaguchi, {\em When is the classifying space for elliptic fibrations rank one?}, Bull. Korean Math. Soc. {\bf 42} (2005), no.~3, 521--525. 


\end{thebibliography}
\end{document}